\documentclass[11pt]{article}
\usepackage[english]{babel}
\usepackage{amssymb,amsmath,graphicx}
\usepackage{amsthm}
\usepackage{float}
\usepackage{afterpage}
\usepackage{color}
\usepackage[numbers]{natbib}
\usepackage{enumitem}
\usepackage{setspace}
\usepackage[hidelinks]{hyperref}
\usepackage{algorithmic}
\usepackage{algorithm}
\usepackage[table]{xcolor}
\usepackage{multirow}
\usepackage{tikz}
\usepackage{lscape}
\usepackage{graphicx}
\usepackage{ltablex}
\graphicspath{ {C:/Users/Ha/Dropbox/Field partitioning problem/Thibaut Version} }
\usepackage{longtable}
\usepackage{microtype}

\newcolumntype{M}[1]{>{\centering}m{#1}}
\newcolumntype{H}{>{\setbox0=\hbox\bgroup}c<{\egroup}@{}}
\usepackage{nicefrac}

 \bibpunct[, ]{(}{)}{,}{a}{}{,}%
 \def\newblock{\ }%

\newtheorem{definition}{Definition}
\newtheorem{theorem}{Theorem}
\newtheorem{lemma}{Lemma}

\newcolumntype{d}[1]{D{.}{.}{#1}}

\newcommand{\cO}{{\mathcal O}}

\allowdisplaybreaks

\title{On three soft rectangle packing problems with the guillotine constraints}

\author{Quoc Trung Bui, Thibaut Vidal, Minh Ho\`ang H\`a}

\begin{document}

\begin{center}

\begin{LARGE}
On three soft rectangle packing problems with guillotine constraints
\end{LARGE}

\vspace*{1cm}

\textbf{Quoc Trung Bui} \\
Daily-Opt Joint Stock Company\\
\href{mailto:trungbui@daily-opt.com}{trungbui@daily-opt.com} \\
\vspace*{0.2cm}
\textbf{Thibaut Vidal} \\
Departamento de Inform\'{a}tica, Pontif\'{i}cia Universidade Cat\'{o}lica do Rio de Janeiro \\
\href{mailto:vidalt@inf.puc-rio.br}{vidalt@inf.puc-rio.br} \\
\vspace*{0.2cm}
\textbf{Minh Ho\`ang H\`a *} \\
University of Engineering and Technology, Vietnam National University\\
\href{mailto:minhhoang.ha@vnu.edu.vn}{minhhoang.ha@vnu.edu.vn} \\
\vspace*{0.4cm}

\end{center}

\noindent
\textbf{Abstract.} We investigate how to partition a rectangular region of length $L_1$ and height $L_2$ into~$n$ rectangles of given areas $(a_1, \dots, a_n)$ using two-stage guillotine cuts, so as to minimize either (i) the sum of the perimeters, (ii) the largest perimeter, or (iii) the maximum aspect ratio of the rectangles. These problems play an important role in the ongoing Vietnamese land-allocation reform, as well as in the optimization of matrix multiplication algorithms. We show that the first problem can be solved to optimality in $\cO(n \log n)$, while the two others are NP-hard. We propose mixed integer programming (MIP) formulations and a binary search-based approach for solving the NP-hard problems. Experimental analyses are conducted to compare the solution approaches in terms of computational efficiency and solution quality, for different objectives.
\vspace*{0.2cm}

\noindent
\textbf{Keywords.} Soft rectangle packing, guillotine constraints, complexity analysis, integer programming.

\thispagestyle{empty}
\pagenumbering{arabic}

\section{Introduction}
\label{sec_introduction}

We consider a family of soft rectangle packing problems in which a rectangular region of length~$L_1$ and height $L_2$ must be partitioned into $n$ rectangles of positive areas $(a_1,\dots,a_n)$, where $\sum_{i=1}^n a_i = L_1 \times L_2$. The areas of the rectangles are fixed, and their position, length and height constitute the decision variables of the problem. 
Three different objectives are considered: minimizing the sum of the rectangle's perimeters, the largest perimeter, or the largest aspect ratio, leading to three problems called \mbox{\textsc{Col-Peri-Sum}}, \textsc{Col-Peri-Max}, and \textsc{Col-Aspect-Ratio} respectively. Finally, the layout of the rectangles is subject to strict rules. As illustrated in Figure \ref{fig:example}, the rectangles should be delimited by \emph{two-stage} guillotine cuts: first cutting the rectangular area horizontally to produce several layers (three on the figure), and then cutting each layer vertically to obtain the rectangles (ten on the figure).

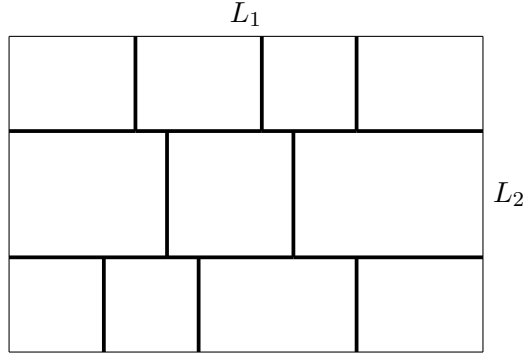
\begin{figure} [htbp]
\centering
\begin{tikzpicture}[scale = 0.42]

\coordinate (N1) at (0, 0);
\coordinate (N2) at (15, 0);
\coordinate (N3) at (15, 10);
\coordinate (N4) at (0, 10);

\coordinate (r1) at (0, 3);
\coordinate (r2) at (15, 3);

\coordinate (r3) at (0, 7);
\coordinate (r4) at (15, 7);

\coordinate (x1) at (3, 0);
\coordinate (y1) at (3, 3);

\coordinate (x2) at (6, 0);
\coordinate (y2) at (6, 3);

\coordinate (x3) at (11, 0);
\coordinate (y3) at (11, 3);

\coordinate (x5) at (5, 3);
\coordinate (y5) at (5, 7);

\coordinate (x6) at (9, 3);
\coordinate (y6) at (9, 7);

\coordinate(x7) at (4, 7);
\coordinate(y7) at (4,10);

\coordinate(x8) at (8, 7);
\coordinate(y8) at (8, 10);

\coordinate(x9) at (11, 7);
\coordinate(y9) at (11, 10);


\path[-] (N3) edge node[anchor = south] {$L_1$} (N4);

\path[-] (x1) edge[ line width=0.5mm]  (y1);
\path[-] (x2) edge[ line width = 0.5mm]  (y2);
\path[-] (x3) edge[ line width = 0.5mm]  (y3);

\path[-] (N1) edge node[anchor = south] {\scriptsize } (x1);
\path[-] (x1) edge node[anchor = south] {\scriptsize } (x2);
\path[-] (x2) edge node[anchor = south] {\scriptsize } (x3);
\path[-] (x3) edge node[anchor = south] {\scriptsize } (N2);
\path[-] (N2) edge node[anchor = west] {\scriptsize } (r2);

\path[-] (x5) edge[ line width = 0.5mm]  (y5);
\path[-] (x6) edge[ line width = 0.5mm]  (y6);

\path[-] (r1) edge[ line width = 0.5mm] node[ anchor = south] {\scriptsize} (x5);
\path[-] (x5) edge[ line width = 0.5mm]  node[anchor = south] {\scriptsize } (x6);
\path[-] (x6) edge[ line width = 0.5mm]  node[anchor = south] {\scriptsize } (r2);
\path[-] (r2) edge  node[anchor = west] {$L_2$} (r4);

\path[-] (x7) edge[ line width = 0.5mm]  (y7);
\path[-] (x8) edge[ line width = 0.5mm]  (y8);
\path[-] (x9) edge[ line width = 0.5mm]  (y9);

\path[-] (r3) edge[ line width = 0.5mm]  node[anchor = south] {\scriptsize } (x7);
\path[-] (x7) edge[ line width = 0.5mm]  node[anchor = south] {\scriptsize } (x8);
\path[-] (x8) edge[ line width = 0.5mm]  node[anchor = south] {\scriptsize } (x9);
\path[-] (x9) edge[ line width = 0.5mm]  node[anchor = south] {\scriptsize } (r4);
\path[-] (r4) edge  node[anchor = west] {\scriptsize } (N3);

\path[-] (N1) edge node[anchor = east] {\scriptsize } (r1);
\path[-] (r1) edge  node[anchor = east] {\scriptsize } (r3);
\path[-] (r3) edge node[anchor = east] {\scriptsize } (N4);

\end{tikzpicture}
\caption{Partitioning the rectangular area by two-stage guillotine cuts -- example solution.}
\label{fig:example}
\end{figure}

Any solution of these problems can be described as a \emph{partition} $\{S_1, \dots, S_m\}$ of the rectangle set $S = \{1,\dots,n\}$ into $m$ layers. Since the length of each layer is fixed to $L_1$, the height $w(S_k)$ of a layer $S_k$ (and therefore of all its contained rectangles) takes value
\begin{equation}
w(S_k) = \frac{\sum_{i \in S_k} a_i} {L_1},
\end{equation}
and the length of each rectangle $i \in S_k$ is $a_i / w(S_k)$.
Based on this observation, the objective of these problems is to find $\{S_1,\dots,S_m\}$ so as to minimize:
\begin{align}
\textsc{Col-Peri-Sum: \hspace*{0.2cm}}  \Phi_1 &= 2 \times \sum_{k=1}^m \sum_{i \in S_k} \left( w(S_k) + \frac{a_i} {w(S_k)} \right) \nonumber \\ 
&=  2  \times \sum_{k=1}^m \left(|S_k| w(S_k) + L_1 \right) \label{obj1} \\[0.5em]
\textsc{Col-Peri-Max: \hspace*{0.2cm}}  \Phi_2 &= 2 \times \max_{k}  \max_{i \in S_k}  \left( w(S_k) + \frac{a_i} {w(S_k)} \right)  \\[0.5em]
\textsc{Col-Aspect-Ratio: \hspace*{0.2cm}}  \Phi_3 &= \max_{k} \max_{i \in S_k}  \max \left(  \frac{a_i} {w(S_k)^2},      \frac {w(S_k)^2}{a_i} \right). \label{obj3}
\end{align}


The main contribution of this paper is to characterize the computational complexity of these problems, proposing an efficient  $\cO(n \log n)$ algorithm for \textsc{Col-Peri-Sum}, and demonstrating that the two other problems are NP-hard. Moreover, we introduce mixed integer programming (MIP) formulations for the NP-hard problems. Finally, we conduct experimental analyzes to determine the limit-size of the instances which can be efficiently solved, and compare the solutions obtained with different objectives.

%

\section{Applications and related work}
\label{moti}

\textbf{Land reform in Vietnam.}
Historically, the agricultural land of Vietnam has been classified into several categories, and each household has been given one plot for each land category, such that even a small agricultural field can be distributed to many households. These division rules have been applied in most provinces of Vietnam to ensure equality among households. However, this has led to a large fragmentation of the land in most provinces of Vietnam \citep{Patrik2006,pham07}. In these provinces, each household owns many small and scattered plots, located in different fields. The province of Vinh Phuc is a striking example: some households have up to 47 plots, and each plot has an average area of only ten square meters \citep{Doe:2009:Online}.

The land fragmentation turned out to be detrimental in the industrialized era. First, households cannot use machines to cultivate small plots, leading to a high production cost. Second, fragmented plots are costly to visit and maintain. Third, the excessive number of tracks separating the plots causes a waste of agricultural land \citep{ref8,ref9,ref7,ref10,Patrik2006}. Therefore, the Vietnamese government considers land fragmentation to be ``a significant barrier to achieving further productivity gains in agriculture'', and initiated a land reform to deal with the situation. This reform aims to reduce the number of land categories, to merge small plots into large fields and finally to repartition these fields into larger plots for households. It has led to successful results in some provinces, as characterized by a significant increase in rice yield attaining 25\% in Quang Nam province \citep{Patrik2006,trung2013}.

The land reform involves two critical tasks: merging small plots into larger fields, and repartitioning these fields equitably while respecting the predefined quantity of land attributed to each household. In this study, we consider the case of rectangular fields, as this is the most common in practice. The fields should be first split by parallel edge-to-edge tracks to facilitate the use of machines, and the resulting sections should then be separated into plots, leading to the two-stage guillotine constraints discussed in the introduction of this paper.

Finally, farmers and local authorities may have distinct objectives and motivations.
The local authorities aim to minimize the amount of wasted land due to the creation of tracks, a goal which is captured by the \textsc{Col-Peri-Sum} objective. In contrast, the farmers wish to have their plots as square as possible to facilitate cultivation. To obtain equitable solutions, this goal can be expressed as a worse-case optimization, leading to the \textsc{Col-Peri-Max} and \textsc{Col-Aspect-Ratio} objectives. These objectives are not strongly conflicting, but they often lead to different land allocation decisions. 

Land-consolidation strategies have been implemented in various other countries, e.g., in Germany \citep{Borgwardt2014,Brieden20014, Borgwardt2011}, Turkey \citep{cay2006, Cay2010262, Cay2013541, Hakli2017}, Japan \citep{Arimoto2010}, Cyprus \citep{Demetriou2013}, China \citep{Huang2011}, and Brazil \citep{Gliesch2017}. However, each country, depending on its own topology, culture, and practice has converged towards a different problem setting. In particular, the two-stage guillotine-cut restrictions and the objective functions relevant to the Vietnamese case have not yet been encountered in other land-consolidation applications. Still, some related problems can be found in the operations research and computer science literature, as discussed in the following.\\

\noindent
\textbf{Soft rectangle packing problems.}
Partitioning an area into polygons of fixed shape or area is a class of problems which has been regularly studied in the operations research and computational geometry literature.
\cite{Beaumont2001,Beaumont2002} defined two optimization problems that seek to partition the unit square into a number of rectangles with given areas, so as to optimize parallel matrix-multiplication algorithms in heterogeneous parallel computing platforms. The first problem aims to minimize the sum of all rectangle perimeters, whereas the second problem aims to minimize the largest perimeter. These problems are special cases of \textsc{Peri-Sum} and \textsc{Peri-Max} where the general rectangular region is a square. The authors introduced an $\nicefrac{7}{4}$-approximate algorithm and an $\nicefrac{2}{\sqrt{3}}$-approximate algorithm for these problems,~respectively.

Later on, \cite{Nagamochi2007523} considered the general \textsc{Peri-Sum}, \textsc{Peri-Max} and \textsc{Aspect-Ratio} problems without guillotine constraints. They introduced an $\cO(n \log n)$-time algorithm which produces a $1.25$-approximate solution for \textsc{Peri-Sum}, a $\nicefrac{2}{\sqrt{3}}$-approximate solution for \textsc{Peri-Max}, and finds a solution with aspect ratio smaller than $\smash{\max\{R, 3, 1 + \hspace*{-0.3cm} \max\limits_{i\in\{1,...,n-1\}}  \hspace*{-0.2cm} \frac{a_{i+1}}{a_i} \}}$ for \textsc{Aspect-Ratio} where $R$ denotes the aspect ratio of the original rectangular area. \cite{Fugenschuh2014} also designed an $\nicefrac{2}{\sqrt{3}}$-approximate algorithm and a branch-and-cut algorithm~for~\textsc{Peri-Sum}.

Other close variants of \textsc{Peri-Sum} have been studied.
\cite{KONG1987239,KONG2} considered the problem of decomposing a square or a rectangle into a number of rectangles of equal area so as to minimize the maximum rectangle perimeter. VLSI floorplan design and facility location applications also led to a number of related studies \mbox{\citep{Young2001,Ji2017,Paes2016}}. 
\cite{Ibaraki} proposed a local search and mathematical programming algorithm to solve rectangular packing problems where the shapes of the rectangles are adjustable within given perimeter and area constraints.

Finally, \cite{Beaumont2002} considered \textsc{Col-Peri-Sum} and \textsc{Col-Peri-Max} as a building block to design approximation algorithms for \textsc{Peri-Sum} and  \textsc{Peri-Max} when the general rectangular region is a square. The authors introduced an exact $\cO(n^2 \log n)$ algorithm for \mbox{\textsc{Col-Peri-Sum}} and two approximation algorithms for \mbox{\textsc{Col-Peri-Max}}. The complexity status of \textsc{Col-Peri-Max} remained open. Moreover, \mbox{\textsc{Col-Aspect-ratio}} has not been studied to this date. These methodological gaps along with the relevance of these problems for the Vietnamese land reform are a strong motivation for additional research.

\section{COL-PERI-SUM can be solved in $\boldsymbol\cO\mathbf{(n \textbf{\,log\,} n)}$}
\label{ComComSec}

A polynomial-time algorithm in $\cO(n^2 \log n)$ for \textsc{Col-Peri-Sum} was proposed in \cite{Beaumont2002}. In this section, we introduce a simple algorithm in~$\cO(n \log n)$ for this problem. To that extent, we show that after ordering the rectangles' indices by non-decreasing area, the \textsc{Col-Peri-Sum} problem can be reduced in $\cO(n)$ to the concave least-weight subsequence problem (CLWS), solvable to optimality in $\cO(n)$ time \citep{Wilber1988418}.

\begin{definition}[\emph{Concave real-value weight function}]
\emph{
A real-value weight function $w(i,j)$ defined for integers $0 \leq i < j \leq n$ is concave if and only if, for $0 \leq i_0 < i_1 < j_0 < j_1 \leq n$, $w(i_0,j_0) + w(i_1,j_1) \leq w(i_0,j_1) + w(i_1,j_0)$.
}
\end{definition}

\begin{definition}[\emph{Concave least-weight subsequence problem}]
\emph{
Let $w(i,j)$ be a concave real-value weight function defined for integers $0 \leq i < j \leq n$. Find an integer $k \geq 1$ and a sequence of integers $0 = l_0 < l_1 < \dots < l_{k-1} < l_k = n$ such that $\sum_{i=0}^{k-1} w(l_i, l_{i+1})$ is minimized.
}
\end{definition}

To do this reduction, we first assume that the indices of the rectangles have been ordered in $\cO(n \log n)$ by non-decreasing area: $a_1 \leq \dots \leq a_n$. 
We now highlight a property of \textsc{Col-Peri-Sum} which allows to focus the search on a smaller subset of solutions.

\begin{theorem}
\emph{
Consider a solution $\mathbf{s}$ of \textsc{Col-Peri-Sum} with cost $\Phi_1(\mathbf{s})$, represented as a partition $\{S_1, \dots, S_m\}$ of the rectangle set. Let $i \in S_{k}$ and $j \in S_{l}$ be two rectangles from different subsets such that $a_{i} > a_{j}$. Any solution $\mathbf{s'}$ obtained by swapping these two rectangles within their respective subsets is such that: 
$$
\begin{cases}
\Phi_1(\mathbf{s'}) < \Phi_1(\mathbf{s})  & \text{if } |S_{k}| > |S_{l}|  \\
\Phi_1(\mathbf{s'}) =  \Phi_1(\mathbf{s})  & \text{if } |S_{k}| = |S_{l}| \\
\Phi_1(\mathbf{s'}) > \Phi_1(\mathbf{s})  & \text{otherwise}
\end{cases}
$$
}
\end{theorem}

\begin{proof}
Simply evaluate the cost difference using Equation (\ref{obj1}):
$$
\begin{aligned}
\Delta &=  \Phi_1(\mathbf{s'}) - \Phi_1(\mathbf{s}) \\
&= \frac{2}{L_1} \left( |S_{k}| \left (a_{j} - a_{i} + \sum_{x \in S_{k}} a_x \right) + |S_{l}| \left (a_{i} - a_{j} + \sum_{x \in S_{l}} a_x \right)  -  |S_{k}| \sum_{x \in S_{k}} a_x  -  |S_{l}| \sum_{x \in S_{l}} a_x    \right)  \\
&= \frac{2}{L_1} \left( |S_{k}| - |S_{l}| \right)  (a_{j} - a_{i}). \hspace*{0.3cm} \qedhere
\end{aligned}
$$
\end{proof}

This theorem defines some important features of the optimal solutions of \textsc{Col-Peri-Sum}:
\begin{itemize}[nosep,leftmargin=0.15in]
\item First, without loss of generality, any solution of \textsc{Col-Peri-Sum} can be presented in such a way that $|S_1| \geq \dots \geq |S_m|$ (re-ordering the subsets according to their cardinality).
\item With this representation, if $k < l$ and $|S_k| = |S_l|$, there exists an optimal solution such that $a_i \leq a_j$ for all $i \in |S_k|$ and $j \in |S_l|$.
\item Finally, if $k < l$ and $|S_k| > |S_l|$, all optimal solutions satisfy $a_i \leq a_j$ for all $i \in |S_k|$ and~$j \in |S_l|$.
\end{itemize}

\noindent
Following from these observations, there exists an optimal solution $\mathbf{s^*} = \{S_1, \dots, S_m\}$ of \textsc{Col-Peri-Sum} such that each $S_k$ for $k \in \{1,\dots,m\}$ is a subsequence (of consecutive indices) of the sequence $\langle a_1, a_2, \dots, a_n \rangle$. 
Therefore, we can find an optimal solution of \textsc{Col-Peri-Sum} by solving a least weight subsequence problem instance over the set of integers $0 \leq i < j \leq n$ with the weight function: $$w_\textsc{P}(i,j) = 2 \left( L_1 + \frac{(j-i)}{L_1}\sum_{k = i+1}^{j}a_k \right),$$
where $w_\textsc{P}(i,j)$ represents the sum of the perimeters of the rectangles of indices $(i+1,\dots,j)$ when positioned in a single layer. Finally, Theorem \ref{concave} proves that this weight function is concave, leading to an instance of CLWS. 

\begin{theorem}
\label{concave}
\emph{
The weight function $w_\textsc{P}(i,j)$ is concave.
}
\end{theorem}
\begin{proof}
For $0 \leq i_0 < i_1 < j_0 < j_1 \leq n$, one can directly verify that: \vspace*{-0.1cm}
$$ 
\begin{aligned}
\Delta' &=  w(i_0,j_1) + w(i_1,j_0) - w(i_0, j_0) - w(i_1,j_1)  \nonumber \\
 &= \frac{2}{L_1} \left( (j_1 - i_0) \sum_{x = i_0+1}^{j_1} a_x  + (j_0- i_1) \sum_{x = i_1+1}^{j_0} a_x - (j_0 - i_0) \sum_{x = i_0+1}^{j_0} a_x  -  (j_1- i_1) \sum_{x = i_1+1}^{j_1} a_x \right)  \nonumber \\
  & =  \frac{2}{L_1} \left( (i_1 - i_0) \sum_{x = j_0+1}^{j_1} a_x +  (j_1 - j_0)\sum_{x = i_0+1}^{i_1} a_x  \right)   > 0. \hspace*{0.5cm} \qedhere
\end{aligned}
$$

\end{proof}

As a consequence, after prior ordering of the rectangles in $\cO(n \log n)$, an optimal solution of \textsc{Col-Peri-Sum} can be found by solving an instance of CLWS, e.g., using the $\cO(n)$ algorithm of \citep{Wilber1988418}. \textsc{Col-Peri-Sum} can thus be solved in $\cO(n \log n)$ in the general case, and in $\cO(n)$ if the rectangles are ordered by non-decreasing (or non-increasing) area in the input.

\section{NP-hardness results}

In the previous section, we have seen that an efficient $\cO(n \log n)$ algorithm can be designed for \textsc{Col-Peri-Sum}. In contrast, we will show in the following that the ``min-max'' version of this problem, \textsc{Col-Peri-Max}, as well as the \textsc{Col-Aspect-Ratio} problems are more difficult.

Let \textsc{Col-Peri-Max-$\Phi$} and \textsc{Col-Aspect-Ratio-$\Phi$} be the decision problems in which one must determine whether there exists a solution of value at most $\Phi$ for \textsc{Col-Peri-Max}, and \textsc{Col-Aspect-Ratio}, respectively. We will show that these two problems are NP-complete, by reduction from \textsc{2-Partition} \citep{Garey:1990:CIG:574848}, hence establishing the NP-hardness of \textsc{Col-Peri-Max} and \textsc{Col-Aspect-Ratio}.

\pagebreak

\begin{theorem}
\emph{
\textsc{Col-Peri-Max-$\Phi$} is NP-complete.
}
\end{theorem}

\begin{proof}
In \textsc{2-Partition}, we are given $n$ positive integers $c_1, \dots, c_n$, and should determine whether there is a partition $S_1 \cup S_2 = \{1,\dots,n\}$, $S_1 \cap S_2 = \varnothing$ such that $\sum_{x \in S_1} c_x = \sum_{x \in S_2} c_x$.

Let $c_\textsc{max} =  \max_{i \in \{1,\dots,n\}} c_i$, and consider the following \textsc{Col-Peri-Max-$\Phi$} instance:
\begin{itemize}[nosep,leftmargin=0.15in]
\item[--] a rectangular area of length $L_1 = \frac{1}{2} \sum_{i = 1}^{n} c_i$ and height $L_2  = 2  c_\textsc{max}$; 
\item[--] for $i \in \{1,\dots,n\}$,  rectangle $i$ has an area $a_i = c_i  \times c_\textsc{max}$;
and
\item[--] $\Phi = 4 \times c_\textsc{max}$.
\end{itemize}


Assume that \textsc{2-Partition} is \textsc{True}: there exists a partition $(S_1,S_2)$ such that $\sum_{x \in S_1} c_x = \sum_{x \in S_2} c_x$. Consider a solution of \textsc{Col-Peri-Max-$\Phi$} in which the set of rectangles has been partitioned with $(S_1,S_2)$ into two layers. Each layer has the same total area, forming a solution in which all rectangles have one side of height $\frac{L_2}{2} = c_\textsc{max}$. With this configuration, the rectangle of largest area has the largest perimeter, equal to $4 \times c_\textsc{max} = \Phi$, and thus \textsc{Col-Peri-Max-$\Phi$}~is~\textsc{True}.

Assume that the \textsc{2-Partition} instance is \textsc{False}.
Consider a solution of \textsc{Col-Peri-Max}, and let $S_k$ be the layer which contains the largest rectangle with area $c_\textsc{max}^2$. The sum of areas in $S_k$ is different from $\frac{L_1 \times L_2}{2}$, and thus the height of this layer is different from $c_\textsc{max}$. Hence, the soft rectangle of area $c_\textsc{max}^2$ is not arranged as a square, its perimeter exceeds $4 \times c_\textsc{max}$, and \textsc{Col-Peri-Max-$\Phi$} is \textsc{False}.
\end{proof}

\begin{theorem}
\emph{
\textsc{Col-Aspect-Ratio-$\Phi$} is NP-complete.
}
\end{theorem}

\begin{proof}
As previously, consider an instance of \textsc{2-Partition} with $n$ positive integers $c_1, \dots, c_n$.
Let $C = \sum_{i=1}^n c_i$ and $M = \frac{2(C + 1)^2}{\min_{i \in \{1,\dots,n\}} c_i}$. Define an instance of \textsc{Col-Aspect-Ratio-$\Phi$} as follows:
\begin{itemize}[nosep,leftmargin=0.15in]
\item[--] a rectangular area of length $L_1 = M + \frac{1}{M} + \frac{C}{2}$ and height $L_2 =  2$;
\item[--] $n$ soft rectangles with areas $c_1,\dots, c_n$ as well as two soft rectangles of area $M$ and two soft rectangles of area $\frac{1}{M}$; and
\item[--] $\Phi = M$.
\end{itemize}

If \textsc{2-Partition} is $\textsc{True}$, there exists a partition $(S_1,S_2)$ such that $\sum_{x \in S_1} c_x = \sum_{x \in S_2} c_x = \frac{C}{2}$. We build a solution of \textsc{Col-Aspect-Ratio} with two layers containing the rectangles of $S_1$~and~$S_2$, respectively, as well as one pair of rectangles of area $M$ and $\frac{1}{M}$ each. Each layer has length $M + \frac{1}{M}+ \frac{C}{2}$ and height $1$. In this configuration, a maximum aspect ratio of~$M$, is jointly attained by the largest and smallest rectangles in each layer, and thus \textsc{Col-Aspect-Ratio-$\Phi$} is $\textsc{True}$.

Now, assume that \textsc{2-Partition} is $\textsc{False}$, and discern three possible classes of solutions:
\begin{itemize}[nosep]
\item Consider a solution of \textsc{Col-Aspect-Ratio} with one layer. The rectangle of size $\frac{1}{M}$ has an aspect ratio of $4M$, which exceeds $\Phi$.
\item Consider a solution of \textsc{Col-Aspect-Ratio} with two or more layers, where at least one layer does not contain a rectangle of size $M$. Let $c$ be the area of the largest element in this layer. Two cases should be discerned:
\begin{itemize}

\item If $c = \frac{1}{M}$, then the layer contains one or two small rectangles of~area~$\frac{1}{M}$, and no other rectangle. The aspect ratio of one such rectangle can be computed as the ratio of its length $l_1 \geq \frac{1}{2} ( M + \frac{1}{M}+ \frac{C}{2})$ over its height $\smash{l_2 \leq 2 \times \frac{ \frac{2}{M}}{2M + \frac{2}{M} + C}}$. As such, $\Phi \geq M \times  \frac{(M + \frac{1}{M}+ \frac{C}{2})^2}{4} > M.$

\item Otherwise, there exists at least one rectangle $c_i$ in the layer, and the total area of the layer does not exceed $C+\frac{2}{M}$. The length $l_1$ and height $l_2$ of the rectangle of area $c_i$ satisfy $l_1 \geq \frac{c_i}{C+\frac{2}{M}}  \times ( M + \frac{1}{M}+ \frac{C}{2})$ and $\smash{l_2 \leq  2 \times \frac{ C+\frac{2}{M}}{2M + \frac{2}{M} + C}}$. As such, 
$$\Phi \geq \frac{c_i (M + \frac{1}{M}+ \frac{C}{2})^2}{(C+\frac{2}{M})^2} > \frac{\text{min}_{i=1}^n c_i \times M^2}{(C+1)^2} = M.$$

\end{itemize}
\item Finally, consider a solution of \textsc{Col-Aspect-Ratio} with two layers, where each layer contains exactly one rectangle of size $M$. Since there is no feasible solution of \textsc{2-Partition}, the total areas of the layers are different (the smaller rectangles cannot re-balance the sum due to their small area $1/M$). In the layer of smallest area, the rectangle of area $M$ has a length $l_1 > M$ and height $l_2 < 1$, and thus an aspect ratio $\Phi = \frac{x}{y} > M$.
\end{itemize}

In all cases, there is no solution with an aspect ratio smaller or equal to $\Phi$, and thus \textsc{Col-Aspect-Ratio-$\Phi$} is $\textsc{False}$.
\end{proof}

\section{Mixed integer programming models}
\label{MixIntProSec}

Since \textsc{Col-Peri-Max} and \textsc{Col-Aspect-Ratio} are NP-hard, this section proposes mixed integer programming formulations of these problems, which can be solved to produce optimal solutions for small and medium scale instances.

These models describe a solution with $n$ layers in which some of the layers can be empty.
We associate one binary variable $x_{ik}$ and one continuous variable $w_{ik}$ for each rectangle $i$ and layer $k$. Variable $x_{ik}$ takes value $1$ if and only if rectangle $i$ belongs to layer $k$, and variable $w_{ik}$ represents the length of the soft rectangle $i$ when placed in layer $k$, and $0$ otherwise. Finally, each binary variable $y_k$ takes value $1$ if layer $k$ is non-empty, and $0$ otherwise.

\subsection{Formulation of COL-PERI-MAX}

The mathematical formulation of \textsc{Col-Peri-Max} is given in Equations (\ref{MIP_0})--(\ref{MIP_15}):
\begin{align}
\label{MIP_0}
\text{minimize} \hspace*{0.2cm}& \Phi_2 \\
\label{MIP_1}
\text{s.t.} \hspace*{0.2cm} & 2(L_1+L_2)(x_{ik} - 1) + 2 \left( w_{ik} + \sum_{j = 1}^{n}\frac{a_jx_{jk}}{L_1} \right)  \leq \Phi_2 & i, k \in \{1, \dots, n\}
\\
\label{MIP_2}
&\sum_{k = 1}^{n} x_{ik} = 1 & i  \in \{1, \dots, n\}
\\
\label{MIP_3}
&\sum_{i = 1}^{n}x_{ik} \geq y_k  & k \in \{1, \dots, n\}
\\
\label{MIP_4}
& 
 x_{ik} \leq y_k  & i,k \in \{1, \dots, n\}
\\
\label{MIP_41}
& \sum_{i = 1}^{n}w_{ik} = L_1 y_k & k \in \{1, \dots, n\}
\\
\label{MIP_5}
& w_{ik} \leq L_1 x_{ik} & i,k \in \{1, \dots, n\}
\\
\label{MIP_6}
&a_i x_{ik} \leq L_2 w_{ik}  & i,k \in \{1, \dots, n\}
\\
\label{MIP_7}
&a_j w_{ik}  -  a_i w_{jk}  \leq a_j L_1(2 - x_{ik} - x_{jk})  & i,j,k \in \{1, \dots, n\}, i \neq j
\\
\label{MIP_8}
& a_i w_{jk}  -  a_j w_{ik} \leq a_i L_1 (2 - x_{ik} - x_{jk})  & i,j,k \in \{1, \dots, n\}, i \neq j
\\
\label{MIP_12}
&x_{ji} \in \{0,1\} &  i, j \in \{1, \dots, n\}
\\
\label{MIP_13}
&w_{ik} \geq 0 &  i,k \in \{1, \dots, n\}
\\
\label{MIP_15}
&y_{k} \in \{0,1\} & k \in \{1, \dots, n\}
\\
\label{MIP_16}
& \Phi_2 \geq 0
\end{align} 

Constraints (\ref{MIP_2})--(\ref{MIP_4}) ensure that every rectangle is included in a layer and that $y_k$ takes value~$1$ when at least one rectangle is included in layer $k$.
Constraints (\ref{MIP_41}) state that the sum of the length of the rectangles of each layer $k$ equals $L_1$ if this layer is used ($y_k = 1$), and~$0$ otherwise.
Constraints (\ref{MIP_5}) and (\ref{MIP_6}) impose that $(w_{ik} = 0) \Leftrightarrow (x_{ik} = 0)$.
Finally, Constraints (\ref{MIP_7}) and (\ref{MIP_8}) ensure that if two rectangles $i$ and $j$ are in the same layer $k$, then $a_i / w_{ik} = a_j / w_{jk}$.

This formulation can be strengthened with the addition of some simple optimality cuts which eliminate symmetrical solutions:
\begin{align}
& y_k \geq y_{k+1}   &  k \in \{ 1, \dots, n-1 \}   \label{MIP_1b}\\
& x_{ik}=0                & i \in  \{ 1,\dots,n \},  k \in \{i+1,\dots, n\}  \label{MIP_2b}
\end{align}
The first set of constraints forces the use of layers according to the order of their indices, while the second set of constraints forces any rectangle~$i$ to belong to a layer of index $k \in \{1,\dots,i\}$.

\subsection{Formulation of COL-ASPECT-RATIO}

The objective function $\Phi_3$ is nonlinear, and we did not find a direct mixed integer programming formulation of \textsc{Col-Aspect-Ratio}. Instead, we propose two alternative approaches to generate optimal solutions for this problem. The first approach relies on a change of objective which leads to a linear formulation returning the same optimal solutions as \textsc{Col-Aspect-Ratio}. The second approach exploits the fact that the decision problem \textsc{Col-Aspect-Ratio-$\Phi$}  can be formulated as a MIP. Solving this subproblem in a binary search allows to solve the original optimization problem.

\subsubsection{First approach -- Change of objective function}

We introduce an alternative objective function $\Phi_4$ for \textsc{Col-Aspect-Ratio} which can be modeled in a linear formulation. This objective function can be computed as:
\begin{align}
\Phi_4 &= \max_{k} \max_{i \in S_k}  \frac{ |w(S_k) - \frac{a_i}{w_(S_k)}|}{\sqrt{a_i}}. \label{obj4}
\end{align}

\noindent
The following lemma will be used to prove the equivalence between the two objectives:
\begin{lemma}
\label{lemma1}
\emph{
Given two soft rectangles $i$ and $j$ with side lengths $(l_i, h_i)$ and $(l_j, h_j)$, we have $$ \frac{\max(l_i,h_i)}{\min (l_i,h_i)} \geq \frac{\max (l_j, h_j)}{\min (l_j, h_j)} \iff \frac{|l_i - h_i|}{\sqrt{l_ih_i}} \geq \frac{|l_j - h_j|}{\sqrt{l_jh_j}}.$$}
\end{lemma}

\begin{proof}
Without loss of generality, we can assume that $l_i \geq h_i$ and $l_j \geq h_j$.
Then,
\begin{align*}
  \frac{\max(l_i,h_i)}{\min(l_i,h_i)} &\geq \frac{\max(l_j, h_j)}{\min(l_j, h_j)}  \\
\iff  \frac{l_i}{h_i} &\geq \frac{l_j}{h_j}  \\
\iff  \sqrt{\frac{l_i}{h_i}} &\geq  \sqrt{\frac{l_j}{h_j}} \\
\iff  \left (\sqrt{\frac{l_i}{h_i}} - \sqrt{\frac{l_j}{h_j}}\right ) &\left (1+ \frac{1}{\sqrt{\frac{l_i}{h_i}\frac{l_j}{h_j}}} \right )  \geq  0 \\
\iff  \sqrt{\frac{l_i}{h_i}} - \sqrt{\frac{h_i}{l_i}} & \geq  \sqrt{\frac{l_j}{h_j}} - \sqrt{\frac{h_j}{l_j}} \\
\iff  \frac{l_i - h_i}{\sqrt{l_ih_i}} &\geq  \frac{l_j - h_j}{\sqrt{l_jh_j}}\\
\iff  \frac{|l_i - h_i|}{\sqrt{l_ih_i}} &\geq  \frac{|l_j - h_j|}{\sqrt{l_jh_j}} \qedhere
\end{align*}
\end{proof}

\begin{theorem} \label{optimal}
\emph{
Let $s_3$ and $s_4$ be two optimal solutions obtained with objectives $\Phi_3$ and $\Phi_4$, respectively. Then, $\Phi_3(s_3) = \Phi_3(s_4)$, $\Phi_4(s_3) = \Phi_4(s_4)$, and $s_3$ and $s_4$ are optimal for the objectives $\Phi_4$ and $\Phi_3$, respectively.
}
\end{theorem}

\begin{proof}
For any solution $s$, as a consequence of Lemma \ref{lemma1}, if $\Phi_4(s)$ attains its minimum for a rectangle $i \in \{1,...,n\}$ of length $l_i$ and height $h_i$, then $\Phi_3(s)$ attains its minimum for the same rectangle, and vice-versa. Therefore $\Phi_4(s) = \frac{|l_i - h_i|}{\sqrt{l_ih_i}}$ and $\Phi_3(s) = \frac{\max(l_i,h_i)}{\min(l_i,h_i)}$.

\noindent
Now, assume that $\Phi_4(s_4)$ and $\Phi_3(s_3)$ attain their minimum for rectangles $x$ and $y$, respectively. Therefore,
$$\Phi_4(s_4) = \frac{|l_x - h_x|}{\sqrt{l_xh_x}}, \Phi_3(s_4) = \frac{\max(l_x, h_x)}{\min(l_x,h_x)},$$
$$\Phi_4(s_3) = \frac{|l_y - h_y|}{\sqrt{l_yh_y}}, \Phi_3(s_3) = \frac{\max(l_y, h_y)}{\min(l_y,h_y)}.$$
Since $s_4$ is an optimal solution for objective $\Phi_4$, $\Phi_4(s_4) \leq \Phi_4(s_3)$ and:
$$\frac{|l_x - h_x|}{\sqrt{l_xh_x}} \leq \frac{|l_y - h_y|}{\sqrt{l_yh_y}}$$
Therefore, as a consequence of Lemma \ref{lemma1}, we have
$$\frac{\max(l_x,h_x)}{\min(l_x,h_x)} \leq \frac{\max(l_y, h_y)}{\min(l_y,h_y)}$$
Similarly, since $s_3$ is an optimal solution for objective $\Phi_3$, $\Phi_3(s_3) \leq \Phi_3(s_4)$ and:
$$\frac{\max(l_y, h_y)}{\min(l_y,h_y)} \leq \frac{\max(l_x,h_x)}{\min(l_x,h_x)}$$
Overall,
$$\Phi_3(s_3) = \frac{\max(l_y, h_y)}{\min(l_y,h_y)} = \frac{\max(l_x,h_x)}{\min(l_x,h_x)} = \Phi_3(s_4),$$
and $s_4$ is an optimal solution for objective $\Phi_3$.
A similar proof shows that $s_3$ is an optimal solution for objective $\Phi_4$.
\end{proof}

Based on this change of objective function, \textsc{Col-Aspect-Ratio} can be formulated as: \vspace*{-0.45cm}
\begin{align}
\min \hspace*{0.2cm}& \Phi \nonumber \\
\text{s.t.} \hspace*{0.2cm} & \text{Constraints (\ref{MIP_2})--(\ref{MIP_12})} \nonumber  \\
\label{M2_1}
&\delta_{ik} + L_1(1-x_{ik}) \geq w_{ik} - \sum_{j = 1}^{n} \frac{a_jx_{jk}}{L_1}  & i,k \in \{1, \dots, n\}
\\
\label{M2_2}
&\delta_{ik} + L_2(1-x_{ik}) \geq -w_{ik} + \sum_{j = 1}^{n} \frac{a_jx_{jk}}{L_1} &  i,k \in \{1, \dots, n\}
\\
\label{M2_3}
&\Phi \geq \frac{\delta_{ik}}{\sqrt{a_i}} &  i,k \in \{1, \dots, n\}
\\
\label{M2_4}
&\delta^k_i \geq 0 &  i,k \in \{1, \dots, n\}.
\end{align}

Solving this formulation to optimality generates an optimal solution of \textsc{Col-Aspect-Ratio}. The value of this solution must be recomputed a-posteriori according to the original objective. The model uses $n^2$ additional continuous variables $\delta_{ik}$, as well as a continuous variable $\Phi$ representing the value of the alternative objective function.
According to Constraints (\ref{M2_1}) and (\ref{M2_2}), if a rectangle $i$ is in layer $S_k$, then $\delta_{ik}=|l_i - h_i|$ where $l_i$ and $h_i$ represent the length and height of the rectangle in the current solution, otherwise $\delta_{ik}=0$.

\subsubsection{Second approach -- Binary search}

Another solution approach consists in modeling the decision problem \textsc{Col-Aspect-Ratio-$\Phi$} as a MIP. In this case, a maximum aspect ratio $\Phi$ is set as a constraint, and the goal is to find a feasible solution. The feasibility model can be written as follows:

\begin{align}
& \text{Constraints (\ref{MIP_2})--(\ref{MIP_12})} \nonumber  \\
& L_1 h_k = \sum_{i=1}^n a_i x_{ik} &  k  \in \{1, \dots, n\} \\
& w_{ik} \leq \Phi h_k &   i,k  \in \{1, \dots, n\} \label{AR-11} \\
& h_k  \leq  \Phi w_{ik} &  i,k  \in \{1, \dots, n\} \label{AR-12} \\
& h_k \in \mathbb{R} &  k \in \{1, \dots, n\}
\end{align}

In this model, each variable $h_k$ for $k \in \{1,\dots,n\}$ stores the height of layer $S_k$, and Constraints~(\ref{AR-11})--(\ref{AR-12}) force the aspect ratio to be no higher than $\Phi$.
To solve the original optimization problem, we do a binary search over $\Phi$ and solve \textsc{Col-Aspect-Ratio-$\Phi$} at each step. The starting interval is set to $[\Phi^\textsc{low},\Phi^\textsc{up}]$ where $\Phi^\textsc{low} = 1$ and $\Phi^\textsc{up} = \Phi_3(s_1)$, where $s_1$ is an optimal solution of \textsc{Col-Peri-Sum} found in $\cO(n \log n)$ time. The binary search stops as soon as $\Phi^\textsc{up} - \Phi^\textsc{low} < 0.01$.

\section{Computational Experiments} 
\label{ComExpSec}

To complete the theoretical results of this article, we conducted computational experiments to evaluate the efficiency of the solution methods for the three problems and compare their solutions. All algorithms were implemented in C++ and the mathematical models were solved with CPLEX version 12.4. The experiments were performed on a single thread of an Intel i7-3615QM 2.3 GHz CPU with 10GB RAM, running Mac OS Sierra version 10.12.6, and subject to a CPU time limit of one hour for each run. 

We randomly generated benchmark instances with $n \in \{10, 15, 20, 25, 30, 35, 40\}$ soft rectangles. These instances are subdivided into three classes, and three instances were generated for each class and size for a total of 63 instances.
\begin{itemize}[nosep]
\item \textbf{Class U} -- The area of each item is sampled in a uniform distribution: $X \sim \mathcal{U}(1, 200)$.
\item \textbf{Class MU} -- The area of each item is sampled in a mixture of three uniform distributions: $X \sim \frac{1}{3} \left[ \mathcal{U}(1, 10) + \mathcal{U}(11, 50) + \mathcal{U}(51, 150) \right]$.
\item \textbf{Class MN} -- The area of each item is sampled in a mixture of three normal distributions: $X \sim \frac{1}{3} \left[ \mathcal{N}(5,2) + \mathcal{N}(25,10) + \mathcal{N}(125,50) \right]$, but another sample is taken whenever the area is larger than 200.
\end{itemize}
Finally, the dimensions of the hard rectangle are generated as follows for each instance.
Let~$A$ be the sum of the areas of the soft rectangles. Length $L_1$ is randomly generated with uniform probability in $\{ \lceil \sqrt{\nicefrac{A}{3}} \rceil, \dots, \lfloor \sqrt{3A} \rfloor  \}$. The length of the other side is set to $L_2 = \lfloor \nicefrac{A}{L_1} \rfloor$. Then, $A - L_1  L_2$ soft rectangles are randomly selected, and the area of each rectangle is reduced by one unit in such a way that, after reduction, the area of the hard rectangle coincides with the sum of the areas of the soft rectangles. All benchmark instances are available at \url{https://w1.cirrelt.ca/~vidalt/en/research-data.html}.

\subsection{Performance analysis}

This section compares the CPU time needed to solve \textsc{Col-Peri-Sum}, \textsc{Col-Peri-Max}, and \textsc{Col-Aspect-Ratio}. As expected, the solution of \textsc{Col-Peri-Sum} in $\cO(n \log n)$ is extremely fast, with a measured CPU time of the order of a few milliseconds for all considered instances, such that we concentrate our analyzes on the mathematical programming algorithm for \textsc{Col-Peri-Max} as well as the reformulation and binary search approaches for \textsc{Col-Aspect-Ratio}. To speed up the solution methods, we always generate the optimal solution of \textsc{Col-Peri-Sum} and use it as an initial feasible solution.

Tables \ref{tableU} to \ref{tableMN} report, for each instance class and algorithm, the number of nodes in the search tree (\textbf{Nodes}), the CPU time in seconds (\textbf{Time}), as well as the best found lower bound (\textbf{LB}) and upper bound (\textbf{UB}). For the reformulation-based approach for \textsc{Col-Aspect-Ratio}, columns \textbf{LB}$_{4}$ and \textbf{UB}$_{4}$ correspond to objective $\Phi_4$, and the value of the primal solution for objective $\Phi_3$ is indicated in column \textbf{UB}$_{3}$. TL in column \textbf{Time} means that the CPU time limit of 3600 seconds has been attained. Finally, for the binary search approach  for \textsc{Col-Aspect-Ratio}, we indicate the number of completed iterations in column~\textbf{It}$_\textbf{BS}$.

\begin{table}[htbp]
\renewcommand{\arraystretch}{1.2}
\setlength\tabcolsep{8pt}
\caption{Class U -- Performance comparisons} \label{tableU}
\hspace*{-0.9cm}
\scalebox{0.78}
{
\begin{tabular}{|cc|cccc|ccccc|cccc|}
\hline
\multicolumn{2}{|c|}{Data} & \multicolumn{4}{c|}{\textsc{Col-Peri-Max}} & \multicolumn{5}{c|}{\textsc{Col-Aspect-Ratio} -- Reformulation} & \multicolumn{4}{c|}{\textsc{Col-Aspect-Ratio} -- B. Search}\\ 
\textbf{\#} & \textbf{Size} & \textbf{Nodes} & \textbf{Time} & \textbf{LB} & \textbf{UB} & \textbf{Nodes} & \textbf{Time} & \textbf{LB}$_{4}$ & \textbf{UB}$_{4}$ & \textbf{UB}$_{3}$ & \textbf{It}$_\textbf{BS}$ & \textbf{Time} & \textbf{LB} & \textbf{UB} \\
\hline
p01&10&547&0.53&52.53&52.53&176&0.39&0.95&0.95&2.51&8&1.21&2.51&2.51\\
p02&10&27&0.20&55.17&55.17&68&0.22&1.72&1.72&4.75&10&1.38&4.75&4.75\\
p03&10&101&0.20&55.54&55.54&1&0.08&1.05&1.05&2.73&8&0.93&2.72&2.73\\
p04&15&5139k&TL&52.27&55.14&69&0.74&0.82&0.82&2.22&8&4.91&2.22&2.22\\
p05&15&38k&92.46&51.38&51.38&193&1.85&0.99&0.99&2.59&9&4.27&2.59&2.60\\
p06&15&8054k&TL&48.63&52.46&640k&479.41&1.98&1.98&5.76&10&560.23&5.75&5.76\\
p07&20&351k&TL&43.27&56.57&308&3.38&2.60&2.60&8.64&12&29.04&8.63&8.64\\
p08&20&2993k&TL&34.76&56.14&79k&202.12&0.47&0.47&1.59&7&174.54&1.59&1.59\\
p09&20&92&3.37&53.71&53.71&3k&12.31&1.46&1.46&3.86&9&35.94&3.85&3.86\\
p10&25&479k&TL&49.00&55.86&80k&851.98&0.74&0.74&2.07&8&679.89&2.06&2.07\\
p11&25&935k&TL&45.82&55.14&116&12.37&1.60&1.60&4.33&7&TL&4.29&4.37\\
p12&25&601k&TL&33.01&53.96&393k&TL&0.81&1.01&2.65&2&TL&2.62&4.25\\
p13&30&335k&TL&25.86&54.70&293k&TL&0.00&0.86&2.30&2&TL&2.20&3.41\\
p14&30&318k&TL&40.41&54.11&179k&TL&0.22&1.79&5.01&3&TL&1.00&7.22\\
p15&30&182k&TL&33.64&53.37&328k&TL&1.71&2.59&8.57&3&TL&1.00&17.06\\
p16&35&389k&TL&40.00&56.43&714k&TL&0.00&1.79&5.00&2&TL&1.00&5.34\\
p17&35&247k&TL&36.01&56.43&150k&TL&0.00&1.47&3.92&0&TL&1.00&8.34\\
p18&35&109k&TL&41.35&55.71&149k&TL&0.34&2.59&8.58&1&TL&1.00&28.70\\
p19&40&74k&TL&19.30&56.43&108k&TL&0.64&1.42&3.76&1&TL&1.00&5.86\\
p20&40&101k&TL&44.48&54.79&54k&TL&0.12&1.09&2.83&1&TL&1.93&2.86\\
p21&40&53k&TL&44.12&56.14&68k&TL&0.05&2.73&9.35&2&TL&1.00&21.33\\
\hline
\end{tabular}
}
\end{table}

\begin{table}[htbp]
\renewcommand{\arraystretch}{1.2}
\setlength\tabcolsep{7.82pt}
\caption{Class MU -- Performance comparisons} \label{tableMU}
\hspace*{-0.9cm}
\scalebox{0.78}
{
\begin{tabular}{|cc|cccc|ccccc|cccc|}
\hline
\multicolumn{2}{|c|}{Data} & \multicolumn{4}{c|}{\textsc{Col-Peri-Max}} & \multicolumn{5}{c|}{\textsc{Col-Aspect-Ratio} -- Reformulation} & \multicolumn{4}{c|}{\textsc{Col-Aspect-Ratio} -- B. Search}\\ 
\textbf{\#} & \textbf{Size} & \textbf{Nodes} & \textbf{Time} & \textbf{LB} & \textbf{UB} & \textbf{Nodes} & \textbf{Time} & \textbf{LB}$_{4}$ & \textbf{UB}$_{4}$ & \textbf{UB}$_{3}$ & \textbf{It}$_\textbf{BS}$ & \textbf{Time} & \textbf{LB} & \textbf{UB} \\
\hline
p01&10&179&0.31&55.29&55.29&17&0.24&1.39&1.39&3.65&9&1.03&3.64&3.65\\
p02&10&44&0.07&55.78&55.78&214&0.18&1.36&1.36&3.57&9&1.38&3.56&3.57\\
p03&10&4k&1.22&52.76&52.76&64&0.18&1.97&1.97&5.71&9&0.72&5.70&5.71\\
p04&15&1&0.49&56.17&56.17&720&1.67&2.20&2.20&6.67&11&12.05&6.66&6.67\\
p05&15&6938k&TL&43.43&51.23&294k&280.22&1.36&1.36&3.57&9&451.13&3.56&3.57\\
p06&15&185k&111.49&54.70&54.70&314&1.87&2.13&2.13&6.39&10&8.82&6.39&6.39\\
p07&20&1943k&TL&40.10&55.86&475k&1885.98&2.09&2.09&6.19&10&1450.33&6.19&6.19\\
p08&20&2223k&TL&28.26&53.81&1787k&TL&0.95&1.05&2.74&2&TL&2.69&4.38\\
p09&20&66k&285.98&55.86&55.86&142k&550.19&2.18&2.18&6.61&12&2355.92&6.61&6.61\\
p10&25&688k&TL&40.79&55.71&281k&TL&0.76&1.33&3.48&3&TL&2.98&3.96\\
p11&25&169k&TL&51.08&54.85&220k&TL&1.47&2.24&6.86&2&TL&1.00&12.49\\
p12&25&331k&TL&47.41&56.17&225k&TL&1.77&2.13&6.39&5&TL&5.72&6.66\\
p13&30&325k&TL&21.43&53.07&318k&TL&0.61&1.40&3.68&1&TL&1.00&7.02\\
p14&30&196k&TL&42.46&55.43&176k&TL&1.28&1.96&5.67&2&TL&3.69&6.38\\
p15&30&267k&TL&32.40&55.43&150k&TL&0.91&1.86&5.27&2&TL&1.00&7.10\\
p16&35&115k&TL&41.47&55.86&62k&TL&0.27&2.06&6.09&1&TL&4.94&8.89\\
p17&35&250k&TL&23.28&44.36&5k&297.25&1.39&1.39&3.67&9&471.57&3.67&3.67\\
p18&35&226k&TL&32.23&55.28&143k&TL&0.04&1.96&5.68&1&TL&1.00&7.20\\
p19&40&204k&TL&21.41&54.70&7k&181.11&1.48&1.48&3.92&4&TL&3.87&4.08\\
p20&40&290k&TL&29.09&56.00&72k&TL&0.11&1.53&4.11&3&TL&3.88&5.32\\
p21&40&154k&TL&15.76&52.76&75k&TL&0.00&1.30&3.39&0&TL&1.00&4.60\\
\hline
\end{tabular}
}
\end{table}

\begin{table}[htbp]
\renewcommand{\arraystretch}{1.2}
\setlength\tabcolsep{8pt}
\caption{Class MN -- Performance comparisons} \label{tableMN}
\hspace*{-0.9cm}
\scalebox{0.78}
{
\begin{tabular}{|cc|cccc|ccccc|cccc|}
\hline
\multicolumn{2}{|c|}{Data} & \multicolumn{4}{c|}{\textsc{Col-Peri-Max}} & \multicolumn{5}{c|}{\textsc{Col-Aspect-Ratio} -- Reformulation} & \multicolumn{4}{c|}{\textsc{Col-Aspect-Ratio} -- B. Search}\\ 
\textbf{\#} & \textbf{Size} & \textbf{Nodes} & \textbf{Time} & \textbf{LB} & \textbf{UB} & \textbf{Nodes} & \textbf{Time} & \textbf{LB}$_{4}$ & \textbf{UB}$_{4}$ & \textbf{UB}$_{3}$ & \textbf{It}$_\textbf{BS}$ & \textbf{Time} & \textbf{LB} & \textbf{UB} \\
\hline
p01&10&386&0.31&50.83&50.83&10&0.27&1.05&1.05&2.75&9&0.87&2.74&2.75\\
p02&10&62&0.10&51.50&51.50&1&0.06&0.82&0.82&2.22&7&0.41&2.21&2.22\\
p03&10&267&0.27&51.00&51.00&33&0.28&1.94&1.94&5.60&9&0.75&5.59&5.60\\
p04&15&38k&20.37&39.80&39.80&13k&18.08&1.55&1.55&4.17&9&147.71&4.17&4.17\\
p05&15&130k&94.84&40.99&40.99&2k&5.06&1.03&1.03&2.68&8&20.51&2.67&2.68\\
p06&15&4971k&TL&44.67&45.96&27&0.86&1.65&1.65&4.49&9&5.18&4.49&4.50\\
p07&20&364&9.13&53.17&53.17&23k&200.99&2.82&2.82&9.83&11&2866.39&9.83&9.83\\
p08&20&6k&35.00&53.38&53.38&31k&484.53&2.06&2.06&6.09&10&259.48&6.08&6.09\\
p09&20&1427k&TL&36.61&54.55&59k&290.84&1.05&1.05&2.73&8&554.92&2.73&2.73\\
p10&25&2188k&TL&29.04&43.08&987k&TL&0.56&0.67&1.94&4&TL&1.76&1.95\\
p11&25&233k&TL&40.91&54.55&156&11.20&2.49&2.49&8.09&10&76.60&8.08&8.09\\
p12&25&1341k&1699.02&53.67&53.67&429k&TL&0.55&0.93&2.45&8&661.27&2.44&2.45\\
p13&30&7k&139.13&54.12&54.12&285k&TL&1.39&2.08&6.18&5&TL&5.47&6.21\\
p14&30&843k&TL&35.00&54.41&316k&TL&0.80&1.18&3.06&2&TL&2.28&3.55\\
p15&30&3k&42.77&51.23&51.23&105k&TL&1.85&2.81&9.81&2&TL&6.85&12.69\\
p16&35&113k&TL&39.18&55.43&87k&TL&0.14&1.82&5.12&1&TL&1.00&14.14\\
p17&35&1526k&TL&40.41&53.33&272k&TL&0.14&0.56&1.73&2&TL&1.55&1.74\\
p18&35&169k&TL&44.24&52.57&80k&TL&0.55&1.91&5.48&1&TL&1.00&9.88\\
p19&40&137k&TL&34.26&56.43&49k&TL&0.59&1.93&5.53&0&TL&1.00&6.92\\
p20&40&294k&TL&30.57&55.28&251k&TL&0.84&1.00&2.62&3&TL&2.31&2.64\\
p21&40&141k&TL&9.94&54.55&137k&TL&1.09&1.71&4.70&0&TL&1.00&5.09\\
\hline
\end{tabular}
}
\end{table}

As observed in these experiments, the proposed MIP models can be solved to optimality for all benchmark instances with 10 soft rectangles, as well as a few instances with up to 30 rectangles for \textsc{Col-Peri-Max} and 40 rectangles for \textsc{Col-Aspect-Ratio}. Yet, the number of search nodes and CPU time grow very quickly with the number of soft rectangles $n$.
Despite the symmetry-breaking inequalities, some instances with 15 rectangles lead to over a million search nodes.
The reformulation approach and the binary search approach for \textsc{Col-Aspect-Ratio} find 31/63 and 30/63 optimal solutions, respectively.
The reformulation approach is generally faster than the binary search algorithm for small instances. Yet, a drawback of this algorithm is that it searches for an optimal solution according to objective $\Phi_4$. When optimality is attained, this solution is optimal for $\Phi_3$ due to Theorem \ref{optimal}. When an optimality gap remains, the primal solution obtained from the algorithm gives a valid upper bound for objective $\Phi_3$, but the dual information (and performance guarantee) is lost. Finally, we did not observe a significant difference of performance when comparing the results of the three instance classes (\textbf{U},\textbf{MU} and \textbf{MN}). We noted that two larger instances with 35 and 40 rectangles were solved to optimality for class \textbf{MU}, a phenomenon which did not happen for \textbf{U} and \textbf{MN}.


In general, the difficulties encountered when solving instances with 10 to 40 rectangles already show the limitations of available mathematical programming algorithms for \textsc{Col-Peri-Max} and \textsc{Col-Aspect-Ratio}. Future progress on exact approaches for NP-hard problems may allow to solve larger instances in the future, but years of research may be needed before handling more realistic instances with over a hundred rectangles. Alternatively, heuristics and metaheuristics could be used to solve larger problems. As we noted a complexity gap between \textsc{Col-Peri-Sum} and the other two problems, in spite of the relations between the three objectives, we are interested to see if the solution of \textsc{Col-Peri-Sum} can constitute a viable heuristic for the two more difficult objectives. This is the focus of the next section.

\subsection{Solution evaluations in relation to other objectives}
\label{last-exp}

The objective functions of the three considered problems are different but not strongly conflicting. Yet, \textsc{Col-Peri-Sum} can be solved in $\cO(n \log n)$, while \textsc{Col-Peri-Max} and \textsc{Col-Aspect-Ratio} are NP-hard. In this last analysis, we investigate how \emph{close} these problems are from each other in practice. This is achieved by evaluating the optimal solution of one problem according to the objective function of each other. In particular, we are interested to see if the solution of \textsc{Col-Peri-Sum} can be used as a simple heuristic for \textsc{Col-Peri-Max} and \textsc{Col-Aspect-Ratio}.

For this analysis, we gathered all instances that are solved to optimality for all three problems:
all instances with $n=10$; instances U-p05, MU-p04, MU-p06, MN-p04, and MN-p05 with $n=15$; and instances U-p09, MU-p09, MN-p07, and MN-p08 with $n=20$. For each objective~$\Phi_x$, we evaluated the quality $\Phi_y(s^*_x)$ of its optimal solution $s^*_x$ relatively to each other objective $y \in \{1,2,3\}$, and report the results as the performance ratio $\Phi_y(s^*_x) / \Phi_y(s^*_y)$ in Table~\ref{evalOthers}.

\begin{table}[htbp]
\renewcommand{\arraystretch}{1.2}
\setlength\tabcolsep{10pt}
\caption{Optimal solutions for one objective evaluated according to the other objectives} \label{evalOthers}
\centering
\scalebox{0.85}
{
\begin{tabular}{|c|ccc|ccc|ccc|}
\hline
\textbf{Evaluated as}&\multicolumn{3}{c|}{\textsc{Col-Peri-Sum} -- $\boldsymbol\Phi_{\boldsymbol1}$}&\multicolumn{3}{c|}{ \textsc{Col-Peri-Max} -- $\boldsymbol\Phi_{\boldsymbol2}$}&\multicolumn{3}{c|}{\textsc{Col-Aspect-Ratio} -- $\boldsymbol\Phi_{\boldsymbol3}$}\\
\textbf{Solved as}&$\boldsymbol\Phi_{\boldsymbol1}$&$\boldsymbol\Phi_{\boldsymbol2}$&$\boldsymbol\Phi_{\boldsymbol3}$&$\boldsymbol\Phi_{\boldsymbol1}$&$\boldsymbol\Phi_{\boldsymbol2}$&$\boldsymbol\Phi_{\boldsymbol3}$&$\boldsymbol\Phi_{\boldsymbol1}$&$\boldsymbol\Phi_{\boldsymbol2}$&$\boldsymbol\Phi_{\boldsymbol3}$\\
\hline
MN-p01&1.00&1.30&1.00&1.12&1.00&1.07&1.60&13.09&1.00\\
MN-p02&1.00&1.11&1.00&1.00&1.00&1.00&1.00&22.02&1.00\\
MN-p03&1.00&1.55&1.00&1.04&1.00&1.04&1.00&12.86&1.00\\
U-p01&1.00&1.06&1.00&1.08&1.00&1.09&1.01&4.27&1.00\\
U-p02&1.00&1.00&1.13&1.00&1.00&1.31&1.60&1.60&1.00\\
U-p03&1.00&1.00&1.00&1.00&1.00&1.00&1.00&1.00&1.00\\
MU-p01&1.00&1.23&1.05&1.05&1.00&1.13&1.07&5.97&1.00\\
MU-p02&1.00&1.06&1.01&1.00&1.00&1.00&1.00&22.69&1.00\\
MU-p03&1.00&1.19&1.00&1.05&1.00&1.05&1.00&7.09&1.00\\
MN-p04&1.00&1.41&1.08&1.00&1.00&1.09&1.08&38.85&1.00\\
MN-p05&1.00&1.24&1.02&1.00&1.00&1.04&1.00&9.76&1.00\\
U-p05&1.00&1.04&1.02&1.01&1.00&1.03&1.67&2.31&1.00\\
MU-p04&1.00&1.11&1.02&1.13&1.00&1.06&2.24&21.59&1.00\\
MU-p06&1.00&1.27&1.01&1.00&1.00&1.00&1.24&29.31&1.00\\
MN-p07&1.00&1.15&1.25&1.24&1.00&1.55&1.16&14.65&1.00\\
MN-p08&1.00&1.23&1.14&1.15&1.00&1.33&1.22&9.31&1.00\\
U-p09&1.00&1.06&1.00&1.06&1.00&1.10&1.37&6.35&1.00\\
MU-p09&1.00&1.24&1.04&1.00&1.00&1.07&5.74&29.47&1.00\\
\hline
\textbf{Average}&\textbf{1.00}&\textbf{1.18}&\textbf{1.04}&\textbf{1.05}&\textbf{1.00}&\textbf{1.11}&\textbf{1.50}&\textbf{14.01}&\textbf{1.00}\\
\hline
\end{tabular}
}
\end{table}

These experiments first confirm the fact that the three objectives produce significantly different solutions. For these instances, the optimal solutions of \textsc{Col-Peri-Sum} are within an average factor of $1.05$ of the optimal solutions of \textsc{Col-Peri-Max} when evaluated according to objective $\Phi_2$, and are better than the optimal solutions of \textsc{Col-Aspect-Ratio} with a factor of 1.11. Similarly, the optimal solutions of \textsc{Col-Peri-Sum} give a better approximation of the optimal solutions of \textsc{Col-Aspect-Ratio} than the optimal solutions of \textsc{Col-Peri-Max} (with a factor of 1.50 compared to 14.01).

One likely explanation of these observations is that the objective of \textsc{Col-Peri-Max} mainly concentrates the optimization on rectangles of large area, so as to minimize their perimeter. In \textsc{Col-Peri-Max}, small rectangles almost never play a role as they are unlikely to realize the maximum. In \textsc{Col-Aspect-Ratio}, in contrast, small and large rectangles are equally important, since the maximum aspect ratio can be attained regardless of the rectangle area. Finally, \textsc{Col-Peri-Sum} must optimize the perimeter of all rectangles, regardless of their area, so as to minimize the total sum. This objective leads to optimal solutions which tend to have good overall aspect ratios, regardless of rectangle size.

Finally, a precise analysis of the solutions shows that \textsc{Col-Peri-Sum} produced five optimal solutions for \textsc{Col-Peri-Max} and six optimal solutions for \textsc{Col-Aspect-Ratio} over 18 instances. In one exceptional case (instance p03 of class U), the three methods converged towards the same optimal solution. This situation happened because the optimal solution contained a single layer, but other situations can lead to this behavior: e.g., if a feasible solution exists in which all soft rectangles take the shape of a square, then this solution is indeed optimal for the three objectives.

\section{Conclusions}
\label{ConSec}

In this paper, we investigated three soft rectangle packing problems: \textsc{Col-Peri-Sum}, \textsc{Col-Peri-Max} and \textsc{Col-Aspect-Ratio}. The effective resolution of these problems is of foremost importance for the ongoing land-allocation reform in Vietnam. The objectives considered in these problems model different aspects of fairness and wasted-land minimization.
We introduced an $\cO(n \log n)$ exact algorithm for \textsc{Col-Peri-Sum}. Then, we demonstrated that the two others problems are NP-hard, and proposed compact MIP formulations to solve them. In the case of \textsc{Col-Aspect-Ratio}, an objective reformulation and a binary search scheme were proposed to overcome non-linearities. Through a set of experimental analyzes on 63 benchmark instances, we observed that the resolution of the MIP formulations is currently practicable for problem instances involving 10 to 40 soft rectangles. To solve larger instances of \textsc{Col-Peri-Max} and \textsc{Col-Aspect-Ratio}, we may use the $\cO(n \log n)$-time algorithm for \textsc{Col-Peri-Sum} as a simple heuristic.

The research perspectives are numerous. The proposed MIP formulations can possibly be improved with additional valid inequalities or optimality cuts, and the set-partitioning formulation of the problem can certainly be exploited to develop efficient branch-and-price algorithms. Metaheuristics could also be developed to provide solutions for larger instances or integrate additional restrictions or objectives. Finally, whether \textsc{Col-Peri-Max} and \textsc{Col-Aspect-Ratio} are \emph{strongly} NP-hard remains an interesting open question.


\end{document}